\documentclass[11pt,oneside]{article}
\usepackage[latin1]{inputenc}
\usepackage[T1]{fontenc}
\usepackage{ae}
\usepackage{aecompl}
\usepackage{amsmath}
\usepackage{amsthm}
\usepackage{amsfonts}
\usepackage{amssymb}
\usepackage{mathrsfs}
\usepackage[dvips]{graphicx}
\usepackage{float}
\usepackage[all]{xy}
\usepackage{color}
\theoremstyle{plain}
\newtheorem{prop}{Proposition}[section]
\newtheorem{teo}[prop]{Theorem}
\newtheorem*{teo2}{Theorem 5.3}
\newtheorem*{prop2}{Proposition 4.1}
\newtheorem{defn}[prop]{Definition}

\newtheorem{lem}[prop]{Lemma}

\theoremstyle{remark}
\newtheorem{oss}[prop]{Remark}
\definecolor{red}{rgb}{1,0,0}
\definecolor{green}{rgb}{0,1,0.2}

\title{Uniruledness of some moduli spaces of stable pointed curves}
\author{L. BENZO}
\date{}
\begin{document}
\maketitle
\begin{abstract}
\noindent We prove uniruledness of some moduli spaces $\overline{\mathcal{M}}_{g,n}$ of stable curves of genus $g$ with $n$ marked points using linear systems on nonsingular projective surfaces containing the general curve of genus $g$. Precisely we show that $\overline{\mathcal{M}}_{g,n}$ is uniruled for $g=12$ and $n \leq 5$, $g=13$ and $n \leq 3$, $g=15$ and $n \leq 2$.\\
We then prove that the pointed hyperelliptic locus $\mathcal{H}_{g,n}$ is uniruled for $g \geq 2$ and $n \leq 4g+4$.\\
In the last part we show that a nonsingular complete intersection surface does not carry a linear system containing the general curve of genus $g \geq 16$ and if it carries a linear system containing the general curve of genus $12 \leq g \leq 15$, then it is canonical.
\end{abstract}
\begin{section}{Introduction and overview of the strategy}
Let $\mathcal{M}_{g,n}$ be the (coarse) moduli space of smooth curves of genus $g$ with $n$ marked points defined over the complex numbers and $\overline{\mathcal{M}}_{g,n}$ its Deligne-Mumford compactification.\\
The birational geometry of the moduli spaces $\overline{\mathcal{M}}_{g,n}$ has been extensively studied in the last decade. In 2003 Logan (\cite{LO}) proved the existence of an integer $f(g)$ for all $4 \leq g \leq 23$ such that $\overline{\mathcal{M}}_{g,n}$ is of general type for $n \geq f(g)$. Logan's results were improved by Farkas in \cite{F1}, and by Farkas and Verra in \cite{FV1}.\\
On the other hand various results concerning rationality and unirationality in the range $2 \leq g \leq 12$ were proved by Bini and Fontanari (\cite{BF}), Casnati and Fontanari (\cite{CF}) and Ballico, Casnati and Fontanari (\cite{BCF}). Farkas and Verra (\cite{FV2}) advanced our knowledge in the case of negative Kodaira dimension by proving uniruledness for some $\overline{\mathcal{M}}_{g,n}$, $5 \leq g \leq 10$. Further results in the range $12 \leq g \leq 14$ were proved in \cite{BFV} and \cite{BV}.\\
The methods involved in all these works are not uniform. For example the arguments used by the cited Italian authors recall sometimes classical constructions and in any case do not use computation of classes of divisors in $\overline{\mathcal{M}}_{g,n}$, which is the heart of Logan's method.\\
The following table sums up the results of the cited works, exhibiting what has been proved about rationality, unirationality and uniruledness properties for the $\overline{\mathcal{M}}_{g,n}$ and about their Kodaira dimension.\\
In particular in the above contributions it is proved that $\overline{\mathcal{M}}_{g,n}$ is rational for $0 \leq n \leq a(g)$, unirational for $0 \leq n \leq b(g)$, uniruled for $0 \leq n \leq \sigma(g)$ and has nonnegative Kodaira dimension for $n \geq \tau(g)$, where the values of $a$, $b$, $\sigma$ and $\tau$ are as in the table.\\
Note that for $g=4,5,6,7,11$ one has $\tau(g)=\sigma(g)+1$.
\begin{center}
{\scriptsize
\begin{tabular}{|c|c|c|c|c|c|c|c|c|c|c|c|c|c|c|c|c|c|c|c|c|c|}
  \hline
\label{tabella}
  $g$ & 2 & 3 & 4 & 5 & 6 & 7 & 8 & 9 & 10 & 11 & 12 & 13 & 14 & 15 & 16 & 17 & 18 & 19 & 20 & 21 \\
  \hline
  $a(g)$ & 12 & 14 & 15 & 12 & 8 &  &  &  &  &  &  &  &  &  &  &  &  & & & \\
  \hline
  $b(g)$ & 12 & 14 & 15 & 12 & 15 & 11 & 8 & 9 & 3 & 10 & 1 & 0 & 2 &  &  &  &  & & & \\
  \hline
  $\sigma(g)$ & 12 & 14 & 15 & 14 & 15 & 13 & 12 & 10 & 9 & 10 & 1 & 0 & 2 & 0 & 0 &  &  & & & \\
  \hline
  $\tau(g)$   &  &  & 16 & 15 & 16 & 14 & 14 & 13 & 11 & 11& 11 & 11 & 10 & 10 & 9 & 9 & 9 & 7 & 6 & 4\\
  \hline
\end{tabular}}
\end{center}
$\left.\right.$\\
$\left.\right.$\\
$\left.\right.$
With a little abuse of notation we allowed $n=0$ to refer to the moduli space $\overline{\mathcal{M}}_g$. In the sequel, when we write $\overline{\mathcal{M}}_{g,n}$, we will always suppose $n \geq 1$.\\
Let $g \geq 2$ be a fixed integer and suppose that there exists a nonsingular projective surface carrying a positive-dimensional non-isotrivial linear system containing the general curve of genus $g$. The idea which will be developed in the first part of this work is to use these linear systems to exhibit a rational curve on $\overline{\mathcal{M}}_{g,n}$ (for some $n$) passing through the general point, thus proving the uniruledness of the space.\\
As the table shows, the problem of stating for which pairs $(g,n)$ $\overline{\mathcal{M}}_{g,n}$ is uniruled is almost solved for $4 \leq g \leq 11$. In this range the question is still open only for four moduli spaces, namely $\overline{\mathcal{M}}_{8,13}$, $\overline{\mathcal{M}}_{9,11}$, $\overline{\mathcal{M}}_{9,12}$ and $\overline{\mathcal{M}}_{10,10}$.\\
On the other hand for $17 \leq g \leq 21$ even the Kodaira dimension of $\overline{\mathcal{M}}_g$ is unknown.\\
We will focus our attention on the remaining values $12 \leq g \leq 16$. In this range $\overline{\mathcal{M}}_{12,1}$, $\overline{\mathcal{M}}_{14,1}$ and $\overline{\mathcal{M}}_{14,2}$ were the only moduli spaces of stable pointed curves which were known to be uniruled (actually unirational).\\
Theorem \ref{equivalentimguniruledsistemalineare} assures that if $\overline{\mathcal{M}}_g$ ($g \geq 3$) is uniruled, then a positive-dimensional linear system containing the general curve of genus $g$ on a surface which is not irrational ruled must exist, but does not say anything about how to construct such a linear system.\\
Let us consider for example the case $g=16$. In \cite{CR3} Chang and Ran showed that the class of the canonical divisor $K_{\overline{\mathcal{M}}_{16}}$ is not pseudoeffective, thus proving that $\overline{\mathcal{M}}_{16}$ has Kodaira dimension $-\infty$.
Later it was proved in \cite{BDPP} that a projective variety has a pseudoeffective canonical bundle if and only if it is not uniruled. As a consequence, $\overline{\mathcal{M}}_{16}$ turned out to be uniruled, but the proof, as it is carried out, does not exhibit any linear system as above.\\
Similarly, Chang and Ran showed in \cite{CR2} that $\overline{\mathcal{M}}_{15}$ has Kodaira dimension $-\infty$ by exhibiting a nef curve having negative intersection number with the canonical bundle.\\
This result was then improved in \cite{BV} by Bruno and Verra, which showed the rational connectedness of $\overline{\mathcal{M}}_{15}$ by explicitly constructing a rational curve passing through two general points of the space. In their article a positive-dimensional linear system on a nonsingular canonical surface containing the general curve of genus $15$ is constructed.\\
With a little amount of work, linear systems containing the general curve of genus $12$ and $13$ can be extracted from \cite{V} and \cite{CR}, respectively. The key point here is to show that they can be realized on \emph{nonsingular} projective surfaces.\\
The case $g=14$ can also be handled using \cite{V}, but in this case our method does not improve the known results.\\
Our main theorem is the following
\begin{teo}
\label{maintheorem}
The moduli space $\overline{\mathcal{M}}_{g,n}$ is uniruled for $g=12$ and $n \leq 5$, $g=13$ and $n \leq 3$, $g=15$ and $n \leq 2$.
\end{teo}
The argument used for the proof (Theorem \ref{mgnuniruled}) is in principle generalizable to check uniruledness for various loci inside $\overline{\mathcal{M}}_{g,n}$. As an example we prove the following statement about pointed hyperelliptic loci $\mathcal{H}_{g,n} \subseteq \mathcal{M}_{g,n}$ i.e. loci of points $[(C,p_1,...,p_n)]$ such that $C$ is hyperelliptic:
\begin{prop2}
The $n$-pointed hyperelliptic locus $\mathcal{H}_{g,n}$ is uniruled for all $g \geq 2$ and $n \leq 4g + 4$.
\end{prop2}
In the particular case $g=2$ one has $\mathcal{H}_{2,n}=\mathcal{M}_{2,n}$, thus Proposition \ref{hgn} proves the uniruledness of $\mathcal{M}_{2,n}$ for $n=1,...,12$. Actually $\mathcal{M}_{2,n}$ is known to be rational for $n=1,...,12$ (see \cite{CF}).
Moreover, recently Casnati has proved the rationality of the pointed hyperelliptic loci $\mathcal{H}_{g,n}$ for $g \geq 3$ and $n \leq 2g+8$ (see \cite{CA}).\\ \\
For reasons which will be outlined in the sequel, the natural question about a possible sharpening of these results led us to the study of linear systems containing the general curve of genus $g$ on nonsingular complete intersection surfaces. The following result sums up our conclusions:
\begin{teo2}
Let $C$ be a general curve of genus $g \geq 3$ moving in a positive-dimensional linear system on a nonsingular projective surface $S$ which is a complete intersection. Then $g \leq 15$. Moreover
\begin{itemize}
  \item if $g \leq 11$, then $\emph{Kod}(S) \leq 0$ or $S$ is a canonical surface;
  \item  if $12 \leq g \leq 15$, then $S$ is a canonical surface.
\end{itemize}
\end{teo2}
\end{section}
\begin{section}{Background material}
\label{background}
We work over the field of complex numbers.\\
Let $X$ be a projective nonsingular surface. Consider a fibration $f:X \rightarrow \mathbb{P}^1$ whose general element is a nonsingular connected curve of genus $g \geq 2$, and $n$ disjoint sections $\sigma_i:\mathbb{P}^1 \rightarrow X$ i.e. morphisms such that $f \circ \sigma_i(y)=y$ for all $y \in \mathbb{P}^1$, $i=1,...,n$. Let $E_i \doteq \sigma_i(\mathbb{P}^1)$. Then define the (n+1)-tuple $(f,E_1,...,E_n)$ as the fibration in $n$-pointed curves of genus $g$ whose fibre over $y$ for all $y \in \mathbb{P}^1$ is the $n$-pointed curve $(C,p_1,...,p_n)$, where $C=f^{-1}(y)$ and $p_i \doteq C \cap E_i$.\\
Let $\Psi_{(f,E_1,...,E_n)}: \mathbb{P}^1 \rightarrow \overline{\mathcal{M}}_{g,n}$ be the morphism given (at least on a nonempty open set) by
$$y \mapsto [(f^{-1}(y),p_1,...,p_n)]$$
where $[\left.\right.]$ denotes the isomorphism class of the pointed curve.\\
A fibration is called \emph{isotrivial} if all its nonsingular fibres are mutually isomorphic or, equivalently,
if two general nonsingular fibres are mutually isomorphic. If $f$ is non-isotrivial, then $\Psi_{(f,E_1,...,E_n)}$ defines a rational curve in $\overline{\mathcal{M}}_{g,n}$ passing through the points corresponding to (an open set of) the fibres of $(f,E_1,...,E_n)$.\\ \\
Let $C$ be a nonsingular projective curve of genus $g$ and let $r,d$ be two non-negative integers.
As standard in Brill-Noether theory, we will denote by $W^r_d(C)$ the subvariety of $\text{Pic}^d(C)$ constructed in \cite{ACGH}, having $\text{Supp}(W^r_d(C))=\left\{L \in \text{Pic}^d(C) \left| h^{0}(L) \geq r+1 \right. \right\}$.
We will denote by $\mathcal{W}^r_{g,d}$ the so-called \emph{universal Brill-Noether locus} i.e. the moduli space of pairs $(C,L)$ where $g(C)=g$ and $L \in W^r_d(C)$.\\
A notion which is fundamental for our purposes is that of curve with general moduli.
\begin{defn}
\label{generalcurve}
A connected projective nonsingular curve $C$ of genus $g$ has \emph{general moduli}, or is a \emph{general curve of genus $g$}, if it is a general fibre in a smooth projective family $\mathscr{C} \rightarrow V$ of curves of genus $g$, parameterized by a nonsingular connected algebraic scheme $V$, and having surjective Kodaira-Spencer map at each closed point $v \in V$.
\end{defn}
In the sequel, when we talk about a general curve, we will always assume it to be connected and nonsingular.\\
Let $D$ be a divisor on $C$ and consider the \emph{Petri map} $\mu_0(D):H^{0}(\mathcal{O}_C(D)) \otimes H^{0}(\mathcal{O}_C(K_C-D)) \rightarrow H^{0}(\mathcal{O}_C(K_C))$ given by the cup-product. If $C$ has general moduli, then $\mu_0$ is injective for any $D$. As an easy consequence we obtain the following
\begin{prop}[\cite{AC}, Corollary 5.7]
\label{proprietàcurvemoduligenerali}
Let $C$ be a general curve of genus $g \geq 3$. Then $H^{1}(L^{\otimes 2})=(0)$ for any invertible sheaf $L$ on $C$ such that $h^{0}(L) \geq 2$.
\end{prop}
Coherently with Definition \ref{generalcurve}, we will adopt the following definition:
\begin{defn}
\label{generalcurveonlinearsystem}
Let $S$ be a projective nonsingular surface, and $C \subset S$ a projective nonsingular connected curve $C$ of genus $g \geq 2$. Let $r \doteq \dim(|C|)$. We say that $C$ is a \emph{general curve moving in an $r$-dimensional linear system on $S$} if there is a pointed connected nonsingular algebraic scheme $(V,v)$ and a commutative diagram as follows:
\begin{equation}
\label{famigliadisuperfici}
\xymatrix{C \ar@{^{(}->}[r]  \ar[d] & \mathscr{C} \ar@{^{(}->}[r] \ar[d] & \mathscr{S} \ar[ld]^{\beta} \\ \emph{Spec}(k) \ar[r]^{v} & V & }
\end{equation}
such that
\begin{itemize}
\item[(i)] the left square is a smooth family of deformations of $C$ having surjective Kodaira-Spencer map at every closed point;
\item[(ii)] $\beta$ is a smooth family of projective surfaces and the upper right inclusion restricts over $v$ to the inclusion $C \subset S$;
\item[(iii)] $\dim(|\mathscr{C}(w)|) \geq r$ on the surface $\mathscr{S}(w)$ for all closed points $w \in V$.
\end{itemize}
\end{defn}
Let $S$ be a projective nonsingular surface and let $C \subset S$ be a projective nonsingular connected curve of genus $g$ such that
$\dim(|C|) \geq 1$.
Consider a linear pencil $\Lambda$ contained in $|C|$ whose general member is nonsingular and let $\epsilon:X \rightarrow S$ be the blow-up at its base points (including the infinitely near ones). We obtain a fibration $f:X \rightarrow \mathbb{P}^1$
defined as the composition of $\epsilon$ with the rational map $S \dashrightarrow \mathbb{P}^1$ defined by $\Lambda$. We will call $f$ the \emph{fibration defined by the pencil $\Lambda$}.
\begin{prop}[\cite{S4}, Proposition 4.7]
\label{fibrazioneisotrivialesuperficierigata}
Let $C$ be a general curve of genus $g \geq 3$ moving in a positive-dimensional linear system on a nonsingular projective surface $S$, and let $\Lambda \subset |C|$ be a linear pencil containing $C$ as a member. Then the following are equivalent:
\begin{itemize}
\item[(i)] $\Lambda$ defines an isotrivial fibration;
\item[(ii)] $S$ is birationally equivalent to $C \times \mathbb{P}^1$;
\item[(iii)] $S$ is a non-rational birationally ruled surface.
\end{itemize}
\end{prop}
In the sequel we will say that a linear system is \emph{isotrivial} if a general pencil in it defines an isotrivial fibration.\\
As a corollary of Proposition \ref{fibrazioneisotrivialesuperficierigata} we obtain the following well-known explicit condition for the
uniruledness of $\overline{\mathcal{M}}_g$.
\begin{teo}[\cite{S4}, Theorem 4.9]
\label{equivalentimguniruledsistemalineare}
The following conditions are equivalent for an integer $g \geq 3$:
\begin{itemize}
\item $\overline{\mathcal{M}}_g$ is uniruled;
\item A general curve $C$ of genus $g$ moves in a positive-dimensional linear system on some nonsingular projective surface which is not irrational ruled.
\end{itemize}
\end{teo}
Motivated by this theorem, Sernesi (\cite{S4}) used deformation theory of fibrations over $\mathbb{P}^1$ to work out conditions on a nonsingular projective surface to carry a positive-dimensional non-isotrivial linear system containing the general curve of genus $g$.\\
The following result shows that general curves on irregular surfaces of positive geometric genus cannot move in a positive-dimensional linear
system.
\begin{teo}[\cite{S4}, Theorem 6.5]
\label{sirregolare}
Let $S$ be a projective nonsingular surface with $p_g > 0$ and $q > 0$, and let $C \subset S$ be a nonsingular curve of genus $g \geq 3$. Then $C$ cannot be a general curve moving in a positive-dimensional linear system on $S$.
\end{teo}
Moreover, the following inequality involving the fundamental invariants of the surface has to be satisfied:
\begin{teo}[\cite{S4}, Theorem 5.2]
\label{disuguaglianzafibrazionefree}
Let $C$ be a general curve of genus $g \geq 3$ moving in a positive-dimensional non-isotrivial linear system on a nonsingular projective surface $S$. Then
\begin{equation}
\label{5g-111chios}
5(g-1)-\left[11\chi(\mathcal{O}_S)-2K^2_S+2C^2\right]+h^{0}(\mathcal{O}_S(C)) \leq 0.
\end{equation}
\end{teo}
In the same work a systematical analysis gives upper bounds for the genus of $C$ on $S$ according to the Kodaira dimension of $S$.\\
It turns out that, up to a technical assumption contained in Theorem \ref{superficierazionaleggeq11}, a general curve of genus $g \geq 12$ cannot move in a positive-dimensional non-isotrivial linear system on $S$ if the Kodaira dimension of $S$ is $\leq 0$.\\
Indeed if $\text{Kod}(S) \geq 0$ we have the following theorem:
\begin{teo}[\cite{S4}, Theorem 6.2]
\label{kod=0leq11}
Let $S$ be a projective nonsingular surface with Kodaira dimension $\geq 0$, and let $C \subset S$ be a general curve of genus $g \geq 2$ moving in a positive-dimensional linear system. Then
$$g \leq 5p_g(S)+6+\frac{1}{2}h^{0}(\mathcal{O}_S(K_S-C)).$$
In particular
$$g \leq \left\{\begin{array}{cc}
           6, & p_g=0, \\
           11, & p_g=1.
         \end{array}\right.$$
\end{teo}
Turning to the case $\text{Kod}(S)<0$, Proposition \ref{fibrazioneisotrivialesuperficierigata} tells us that $S$ cannot be a non-rational ruled surface, hence we have only to examine what happens if $S$ is rational.\\
If we have a family (\ref{famigliadisuperfici}) where one fibre of $\beta$ is a rational surface, then all fibres are rational, and by suitably blowing-up sections of $\beta$ (after possibly performing a base change) we can always replace (\ref{famigliadisuperfici}) with a similar family where all fibres of $\beta$ are blow-ups of $\mathbb{P}^2$. Therefore one can always reduce to the situation where there is a birational morphism $\sigma:S \rightarrow \mathbb{P}^2$.
\begin{teo}[\cite{V2}]
\label{superficierazionaleggeq11}
Let $C$ be a general curve of genus $g \geq 2$ moving in a positive-dimensional linear system on a rational surface $S$ such that there is a birational morphism $\sigma:S \rightarrow \mathbb{P}^2$. Suppose that
\begin{itemize}
\item[$(*)$] $|C|$ is mapped by $\sigma$ to a regular linear system of plane curves whose singular points are in general position.
\end{itemize}
Then $g \leq 10$.
\end{teo}
For an extensive discussion of the problems and conjectures related to $(*)$, whose first study dates back to Segre (\cite{SE}), see \cite{V2}, Section 2.2.\\
According to the Enriques-Kodaira birational classification of projective nonsingular surfaces, a surface $S$ with Kodaira dimension $0$ cannot have $p_g(S) \geq 2$, hence, at least under assumption $(*)$, surfaces containing general curves of genus $g \geq 12$ moving in positive-dimensional non-isotrivial linear systems must have Kodaira dimension $\geq 1$.\\ \\
The argument used to prove our main theorem is quite general and can be used to recover the great part of the known results about uniruledness of the $\overline{\mathcal{M}}_{g,n}$. As an example we will use linear systems of curves on $K3$ surfaces.\\
Consider a primitive globally generated and ample line bundle $L$ on a smooth K3 surface $S$. Then the general member of $|L|$ is a smooth irreducible curve (see \cite{SA}). If $g$ is the genus of that curve,
we say that $(S,L)$ is a \emph{smooth primitively polarized K3 surface of genus $g$}. One has $L^2=2g-2$ and $\dim|L|=g$.\\
Assume $g \geq 3$. If $\text{Pic}(S) \cong \mathbb{Z}[L]$, then $L$ is very ample on $S$, hence it embeds $S$ as a smooth projective surface in $\mathbb{P}^g$, whose smooth hyperplane section is a canonical curve (see \cite{SA}).\\
Let $\mathcal{KC}_g$ be the moduli stack parameterizing pairs $(S,C)$, where $S$ is a smooth K3 surface with a primitive polarization $L$ of genus $g$ and $C \in |L|$ a smooth irreducible curve of genus $g$. Let $\mathscr{M}_g$ be the moduli stack of smooth curves of genus $g$, and consider the morphism of stacks
$$c_g:\mathcal{KC}_g \rightarrow \mathscr{M}_g$$
defined as $c_g((S,C))=[C]$.
\begin{teo}[\cite{M1} and \cite{CU}, Proposition 2.2]
\label{cg}
With notation as above, $c_g$ is domi-nant if and only if $g \leq 9$ or $g=11$.
\end{teo}
\end{section}
\begin{section}{Proof of the main theorem}
\label{unirulednessofsomemodulispaces}
To prove Theorem \ref{maintheorem} we state a general result (Theorem \ref{mgnuniruled}) and then apply it to linear systems containing the general curve of genus $12,13$ and $15$, respectively.
\begin{lem}[\cite{FL}, Lemma 2.7]
\label{cd-1ampio}
Let $C$ be a nonsingular projective curve. For all $d \geq 1$, $x \in C$, denote with $C_{{(d-1)},x} \subset C_{(d)}$ the divisor of unordered $d$-tuples containing $x$. $C_{{(d-1)},x}$ is ample in $C_{(d)}$.
\end{lem}
\begin{teo}
\label{mgnuniruled}
Let $C$ be a general curve of genus $g \geq 2$ moving in a non-isotrivial $(r+1)$-dimensional linear system ($r \geq 0$) on a nonsingular projective regular surface $S$, with a deformation $(\mathscr{S},\mathscr{C}) \rightarrow V$ of the pair $(S,C)$ given as in Definition \ref{generalcurveonlinearsystem}.\\
Let $(\mathscr{C},\mathcal{O}_{\mathscr{C}}(\mathscr{C})) \rightarrow V$ be the induced family of deformations of the pair $(C,\mathcal{O}_C(C))$, let $C^2=d$ and suppose that the morphism
$$\alpha: V \rightarrow \mathcal{W}^{r}_{g,d}$$
$$w \mapsto [(\mathscr{C}(w),\mathcal{O}_{\mathscr{C}(w)}(\mathscr{C}(w)))]$$
is dominant.\\
Then $\overline{\mathcal{M}}_{g,n}$ is uniruled for $n \leq r + \rho$, where $\rho=\rho(g,r,d)$ is the Brill-Noether number.
\end{teo}
\begin{proof}
If $r+\rho=0$ there is nothing to prove, hence suppose that $r+\rho \geq 1$.\\
Let $[C]$ be a general point in $\overline{\mathcal{M}}_{g}$ and define $C^r_{(d)}$ to be the inverse image of $W^{r}_{d}(C)$ under the Abel map $C_{(d)} \rightarrow \text{Pic}^d(C)$. We first want to show that there is a $d$-tuple in $C^r_{(d)}$ containing $r+\rho$ general points $q_1,...,q_{r+\rho}$.\\
Since $C_{(d-1),q_i}$ is ample for all $i=1,...,r+\rho$ by Lemma \ref{cd-1ampio}, it intersects every variety of dimension $h$ in a nonempty variety of dimension $\geq h-1$ by Nakai-Moishezon's criterion. Since $C$ has general moduli, $C^r_{(d)}$ has dimension $r+\rho$ and $C_{(d-1),q_1} \cap C^r_{(d)}$ has dimension greater or equal than $r+\rho-1$.
Iterating the above argument, if $r+\rho \geq 2$, one then has that
$$\dim \left(C_{(d-1),q_1} \cap C_{(d-1),q_2} \cap \dots \cap C_{(d-1),q_{r+\rho}} \cap C^r_{(d)}\right) \geq 0$$
hence there is a $d$-tuple in $C^r_{(d)}$ containing the points $q_1,...,q_{r+\rho}$.\\
Since by assumption the map $\alpha$ is dominant and the points $q_i$ are general, there is a point $w \in V$ such that $\mathscr{C}(w) \cong C$ and there is a divisor $D \in |\mathcal{O}_{\mathscr{C}(w)}(\mathscr{C}(w))|$ containing points $r_1,...,r_n$, $n \leq r+\rho$, such that $(C,q_1,...,q_n) \cong (\mathscr{C}(w),r_1,...,r_n)$ as pointed curves.\\
Since $h^1(\mathcal{O}_{\mathscr{S}(w)})=0$, the linear system $|\mathscr{C}(w)|$ cuts out on $\mathscr{C}(w)$ the complete linear series. This means that there exists a non-isotrivial linear pencil $\mathcal{P} \subset |\mathscr{C}(w)|$ whose curves cut on $\mathscr{C}(w)$ the divisor $D$. Now blow up the base points of $\mathcal{P}$ and define $E_i$ as the exceptional divisor over $r_i$ for $i=1,...,n$. This gives a non-isotrivial fibration $(f,E_1,...,E_n)$ over $\mathbb{P}^1$ in $n$-pointed curves of genus $g$, having a curve isomorphic to $(C,q_1,...,q_n)$ among its fibres. Hence there is a rational curve in $\overline{\mathcal{M}}_{g,n}$ passing through the general point $[(C,q_1,...,q_n)]$ of $\overline{\mathcal{M}}_{g,n}$.\\
\end{proof}
\begin{oss}
\label{casorho=0}
If one does not assume the map $\alpha$ to be dominant, Theorem \ref{mgnuniruled} stands true for $n \leq r$. The proof is immediate because, since $\dim(|C|)=r+1$, there is always a linear pencil of curves in $|C|$ passing through $r$ points of $C$.\\
\end{oss}
\begin{oss}
\label{casod=r+rho}
The fact that there is a $d$-tuple of points in $C^r_{(d)}$ containing the points $q_1,...,q_{r+\rho}$ is immediate to see if $d=r+\rho$, since $C_{(d)}$ is irreducible and $C^r_{(d)}$ and $C_{(d)}$ have the same dimension.\\
\end{oss}
\begin{oss}
\label{ggeq7regularredundant}
If $g \geq 7$, the assumption that the surface $S$ is regular is redundant and can consequently be dropped. Let indeed $q=q(S)>0$. If $p_g >0$ too, then $C$ cannot have general moduli by Theorem \ref{sirregolare}. If $p_g=0$, there are two possibilities. If $\text{Kod}(S) <0$, then $S$ is a non-rational ruled surface, thus the linear system $|C|$ is isotrivial by Proposition \ref{fibrazioneisotrivialesuperficierigata}. If $\text{Kod}(S) \geq 0$, then $g \leq 6$ by Theorem \ref{kod=0leq11}, contradiction.\\
\end{oss}
For the sake of simplicity, in the sequel we will sometimes indicate with the same symbol $C$ a general abstract curve of genus $g$ and a projective model of this curve, which will be specified in each situation.
\begin{prop}
\label{m125uniruled}
$\overline{\mathcal{M}}_{12,n}$ is uniruled for $n=1,...,5$.
\end{prop}
\begin{proof}
In \cite{V} the author constructs a reducible nodal curve $D_1 \cup D_2$ of arithmetic genus 12 and degree 17, which is contained in the complete intersection $D_1 \cup D_2 \cup D$ of 5 quadrics in $\mathbb{P}^6$, which is still a reducible nodal curve, and shows that $D_1 \cup D_2$ can be deformed to a general curve $C$ of genus 12. Our first purpose is to show that $C$ is contained in a nonsingular surface which is the complete intersection of four quadrics in $\mathbb{P}^6$.\\
By Lemma 7.6 of \cite{V} one has $h^1(\mathcal{I}_{D_1 \cup D_2}(2))=0$, hence $h^{1}(\mathcal{I}_{C}(2))=0$ by the upper semicontinuity of the cohomology.\\
Since $h^{0}(\mathcal{O}_{\mathbb{P}^6}(2))=28$ and $h^{0}(\mathcal{O}_C(2))=23$, the cohomology sequence associated to the exact sequence
$$0 \rightarrow \mathcal{I}_{C}(2) \rightarrow \mathcal{O}_{\mathbb{P}^6}(2) \rightarrow \mathcal{O}_C(2) \rightarrow 0$$
gives $h^{0}(\mathcal{I}_{C}(2))=5$, hence $C$ is contained in the complete intersection $C \cup C'$ of five quadrics $Q_1,Q_2,...,Q_5$, which we assume to be general quadrics containing $C$.\\
Since $D_1 \cup D_2 \cup D$ is nodal, $C \cup C'$ is nodal too.\\
Note that the base locus of $|\mathcal{I}_C(2)|$ is $C \cup C'$. $Q_i$ is smooth outside of $C \cup C'$ by Bertini's theorem. On the other hand, since for all $p \in C \cup C'$ the tangent space $T_{C \cup C',p}=\bigcap_{i=1}^{5}T_{Q_i,p}$ has dimension 1 or 2, each $Q_i$ is smooth at $p$, hence the general quadric containing $C$ is nonsingular. For the same reason at least four quadrics among the $Q_i$, say $Q_1,...,Q_4$, have to intersect transversally at a fixed point $p \in C \cup C'$, and hence on a nonempty open set of $C \cup C'$, say $(C \cup C') \smallsetminus \left\{p_1,...,p_k\right\}$.\\
Take now $k$ 4-tuples of general quadrics containing $C$, say $(Q^j_1,...,Q^j_4)$, $j=1,...,k$, intersecting transversally, respectively, at $p_j$, and let $s_i$, $i=1,...,4$ (respectively $s^j_i$, $j=1,...,k$) $\in H^{0}(\mathcal{I}_C(2))$ be a section whose zero locus is $Q_i$ (respectively $Q^j_i$). The four quadrics $\widetilde{Q}_i$ defined as the zero loci of the sections $s_i+s^1_i+...+s^k_i \in H^{0}(\mathcal{I}_C(2))$ are four general quadrics containing $C$ and intersecting transversally at every point of $C \cup C'$, hence $S \doteq \bigcap_{i=1}^{4}\widetilde{Q}_i$ is a nonsingular surface containing $C$.\\
A scheme $V$ parameterizing a deformation $\left(\mathscr{C},\mathscr{S}\right)$ as in Theorem \ref{mgnuniruled} can be easily constructed as follows. Consider the  Hilbert scheme of curves of arithmetic genus $12$ and degree $17$ in $\mathbb{P}^6$ and let $\mathscr{H}$ be the open set of smooth curves in the irreducible component containing the general curve.
Let $W \rightarrow \mathscr{H}$ be the $\mathbb{P} \times \mathbb{P} \times \mathbb{P} \times \mathbb{P}$-bundle whose fibre over the point $[C]$ is $\mathbb{P}\left(H^{0}(\mathcal{I}_C(2))\right) \times ... \times \mathbb{P}\left(H^{0}(\mathcal{I}_C(2))\right)$. We can take $V$ to be the open set of $W$ such that the points of the fibres of the restricted projection $V \rightarrow \mathscr{H}$ correspond to nonsingular surfaces.\\
Since $S$ is canonical, the adjunction formula gives
\begin{equation}
\label{relazionecanonica}
\mathcal{O}_C(1) \cong \omega_C(-C)
\end{equation}
hence Riemann-Roch theorem on $\mathcal{O}_C(1)$ and Serre duality give $h^{0}(\mathcal{O}_C(C))=1$, which equals $\dim(|C|)$. The linear pencil $|C|$ is non-isotrivial by Proposition \ref{fibrazioneisotrivialesuperficierigata}.\\
Let $h$ be the $g^6_{17}$ embedding the curve $C$ in $\mathbb{P}^6$. Since, for small deformations of the pair $(C,h)$ in the Hilbert scheme, the linear series $h$ remains very ample and the Petri map remains injective, it follows that $h=|\mathcal{O}_C(1)|$ is a general $g^6_{17}$ i.e. a general point in $W^6_{17}(C)$. Since a complete linear series and its residual have the same Brill-Noether number, relation (\ref{relazionecanonica}) implies that the linear series $|\mathcal{O}_C(C)|$ is a general $g^0_{5}$
and so the map
$$\alpha: V \rightarrow \mathcal{W}^0_{12,5}$$
is dominant.\\
The Brill-Noether number of $|\mathcal{O}_C(C)|$ is $\rho=\rho(12,0,5)=5$. By Theorem \ref{mgnuniruled} the statement follows.
\end{proof}
\begin{prop}
\label{m133uniruled}
$\overline{\mathcal{M}}_{13,n}$ is uniruled for $n=1,2,3$.
\end{prop}
\begin{proof}
In \cite{CR} the authors show that the general curve $C$ of genus 13 can be embedded as a non-degenerate curve of degree 13 in $\mathbb{P}^3$.\\
From the cohomology sequence associated to the exact sequence
$$0 \rightarrow \mathcal{I}_{C}(5) \rightarrow \mathcal{O}_{\mathbb{P}^3}(5) \rightarrow \mathcal{O}_C(5) \rightarrow 0$$
one gets $h^{0}(\mathcal{I}_{C}(5))=3$ since $C$ is of maximal rank by \cite{CR}, Theorem 1. One first has to prove that the generic quintic surface containing $C$ is nonsingular. This is done by explicitly constructing a specialization $C''$ of $C$ lying on a nonsingular quintic and such that $h^{0}(\mathcal{I}_{C''}(5))=h^{0}(\mathcal{I}_{C}(5))$.\\
Let $D$ be a nonsingular rational curve of degree $8$ in $\mathbb{P}^3$ and let $F_4$ be a nonsingular quartic surface containing it.
Consider the linear system $|\mathcal{I}_D(5)|$. One has of course $\text{Bs}|\mathcal{I}_D(5)|=E \subset F_4$ where $E$ is a 1-dimensional scheme. By Bertini's theorem the general quintic containing $D$ is smooth away from $E$. Consider the family of reducible quintics $F_4 \cup A$ where $A$ is a plane. Since for each $p \in E$ the plane $A$ can be always assumed not to contain $p$, there exists a quintic containing $D$ which is smooth at $p$, hence the same is true for the general quintic containing $D$. Now fix $p$ and take such a quintic, say $G$. Since $G$ smooth in $p$, it will be smooth along a nonempty open set of $E$, say $E \smallsetminus \left\{p_1,...,p_k\right\}$. Take $k$ general quintics $G_1,...,G_k$ containing $D$ such that $G_i$ is smooth at $p_i$, and let $s$ (respectively $s_i$, $i=1,...,k$) $\in H^{0}(\mathcal{I}_D(5))$ be a section whose zero locus is $G$ (respectively $G_i$). The quintic defined as the zero locus of the section $s+s_1+...+s_k \in H^{0}(\mathcal{I}_D(5))$ is a general quintic containing $D$ which is smooth along $E$.\\
We can then conclude that the general quintic surface containing $D$ is nonsingular, hence the general element $C'$ of the linear system
$|5H-D|$ on $F_4$, where $H$ is a hyperplane section, is
cut by a nonsingular quintic $F_5$. Since $F_4$ is a K3 surface, $C'$ is a nonsingular irreducible curve of genus $10$ and degree $12$. On $F_5$, the curve $C'$ belongs to the linear system $|4H-D|$, where $H$ is again a hyperplane section. Take a general element $C'' \in |5H-C'|=|H+D|$ on $F_5$. One has $p_a(C'')=13$ and $\deg(C'')=13$. Since the linear subsystem $|H|+D$ consists of reducible curves with singular points which are not fixed, if the curve $C''$ is irreducible, then it will be also nonsingular by Bertini's theorem. Since $h^{0}(\mathcal{I}_{C'}(5)) \geq 5$, one has $\dim |H+D| > \dim|H|+D$ and $C''$ is irreducible.\\
By \cite{IL}, Theorem 3.1, the Hilbert scheme $\mathscr{H}$ parameterizing smooth irreducible curves of genus $13$ and degree $13$ in $\mathbb{P}^3$ is irreducible, hence $C''$ must be a specialization of $C$.
By construction and \cite{R} one has $h^1(\mathcal{I}_{C''}(6-i))=h^1(\mathcal{I}_{C'}(i))=h^1(\mathcal{I}_D(5-i))$ for all $i \in \mathbb{Z}$.
Since the sheaf $\mathcal{I}_D$ is 1-regular in the sense of Castelnuovo-Mumford, it follows that $h^{0}(\mathcal{I}_{C''}(5))=h^{0}(\mathcal{I}_{C}(5))=3$.
But then, since the general quintic surface containing $C''$ is nonsingular, the same is true for the general quintic surface containing $C$.\\
The construction of a scheme $V$ parameterizing a deformation $\left(\mathscr{C},\mathscr{S}\right)$ as in Theorem \ref{mgnuniruled} is analogous to the one done in Proposition \ref{m125uniruled}.\\
Let $S$ be a nonsingular quintic surface containing $C$. Since $S$ is canonical, using (\ref{relazionecanonica}), Riemann-Roch on $\mathcal{O}_C(1)$ and Serre duality one obtains that $|C|$ is a linear system of dimension $3$ on $S$, which is non-isotrivial by Proposition \ref{fibrazioneisotrivialesuperficierigata}.\\
Arguing as in the previous proposition one has that the linear series $|\mathcal{O}_C(1)|$ is a general $g^3_{13}$, $|\mathcal{O}_C(C)|$ is a general $g^2_{11}$ and the map
$$\alpha: V \rightarrow \mathcal{W}^2_{13,11}$$
is dominant.\\
The Brill-Noether number of $|\mathcal{O}_C(C)|$ is $\rho=\rho(13,2,11)=1$. By Theorem \ref{mgnuniruled} the statement follows.
\end{proof}
\begin{prop}
\label{m152uniruled}
$\overline{\mathcal{M}}_{15,n}$ is uniruled for $n=1,2$.
\end{prop}
\begin{proof}
In \cite{BV} the authors show that the general curve $C$ of genus 15 can be embedded as a curve of degree 19 in a nonsingular projective regular surface $S$ which is the complete intersection of four quadrics in $\mathbb{P}^6$. $C$ defines on $S$ a non-isotrivial linear system of dimension $h^{0}(\mathcal{O}_C(C))=2$, in particular the linear series $|\mathcal{O}_C(C)|$ is a $g^1_9$.\\
In the article a family of curves $\mathcal{D}$ whose general element has the above property is considered and it is shown that the morphism
$$u:\mathcal{D} \rightarrow \mathcal{W}^1_{15,9}$$
$$C \mapsto [(C,\omega_C(-1))]$$
is dominant (cf. \cite{BV}, (2.9) p. 5).\\
The construction of a scheme $V$ parameterizing a deformation $\left(\mathscr{C},\mathscr{S}\right)$ as in Theorem \ref{mgnuniruled} is analogous to the one done in Proposition \ref{m125uniruled}.\\
Since $S$ is a canonical surface, for a general $C \in \mathcal{D}$ one has $\omega_C(-1) \cong \mathcal{O}_C(C)$ by adjunction, hence the morphism defined in Theorem \ref{mgnuniruled}
$$\alpha: V \rightarrow \mathcal{W}^{1}_{15,9}$$
is dominant.\\
The Brill-Noether number of the linear series $|\mathcal{O}_C(C)|$ is $\rho=\rho(15,1,9)=1$.\\
By Theorem \ref{mgnuniruled} the statement follows.
\end{proof}
Using Theorem \ref{mgnuniruled} one is able to recover the great part of the known results concerning the uniruledness of moduli spaces $\overline{\mathcal{M}}_{g,n}$ using suitable linear systems of curves. Adam Logan has listed a few in the end of his paper \cite{LO}. Just to give an example one can consider linear systems on K3 surfaces.
\begin{prop}
\label{mgnk3}
$\overline{\mathcal{M}}_{g,n}$ is uniruled for $3 \leq g \leq 9$ or $g=11$ and $n=1,...,g-1$.
\end{prop}
\begin{proof}
Theorem \ref{cg} tells that the hyperplane linear system $|H|$ of a general primi-tively polarized K3 surface $(S,H)$ of genus $g \geq 3$ contains the general curve $C$ of genus $g$ for $g \leq 9$ or $g=11$. One has $h^{0}(\mathcal{O}_C(C))=g$, thus Remark \ref{casorho=0} gives the statement.
\end{proof}
In particular one obtains that $\overline{\mathcal{M}}_{11,n}$ is uniruled for all $n \leq 10$, a result which is sharp (cf. the table).\\
By Theorem \ref{cg} the hyperplane linear system $|H|$ of a general primitively polarized K3 surface $(S,H)$ of genus $10$ does not contain the general curve of genus $10$, thus Theorem \ref{mgnuniruled} cannot be applied to $|H|$.\\
Nevertheless, one can use a result contained in \cite{FKPS} to work out a linear system on which Theorem \ref{mgnuniruled} can be applied. One has the following
\begin{prop}
\label{m10n}
$\overline{\mathcal{M}}_{10,n}$ is uniruled for $n=1,...,7$.
\end{prop}
\begin{proof}
Consider a general primitively polarized K3 surface of genus 11 $(S,H)$. By \cite{FKPS}, section 5, the normalization of the general nodal curve in $|H|$ (having arithmetic genus $p_a(H)=11$) is a general curve of genus $10$. \\
Consider the linear subsystem of curves of $|H|$ having an ordinary node at a general point $p$ of $S$. Blowing up the surface $S$ at $p$ one obtains a linear system $|C|$ whose general element is a general curve of genus $10$, having dimension $\dim|H|-3=8$. Remark \ref{casorho=0} gives the statement.
\end{proof}
\end{section}
\begin{section}{Uniruledness of some pointed hyperelliptic loci $\mathcal{H}_{g,n}$}
The proof of Theorem \ref{mgnuniruled} can in principle be adapted to study the uniruledness of various loci in $\overline{\mathcal{M}}_{g,n}$. The starting point will be to exhibit a positive-dimensional non-isotrivial linear system containing a curve corresponding to the general point of the image of the locus under the morphism to $\overline{\mathcal{M}}_g$ forgetting the marked points.\\
As an example of application of this principle, in the sequel we prove the uniruledness of the pointed hyperelliptic loci $\mathcal{H}_{g,n}$ for all $g \geq 2$ and $n \leq 4g+4$.\\
For each integer $g \geq 2$ consider the linear system $|\mathcal{O}(g+1,2)| \subset \mathbb{P}^1 \times \mathbb{P}^1$. The general element of this system is a general hyperelliptic curve $C$ of genus $g$, whose $g^1_2$ is given by the projection onto the first factor.
Note that, since there is not a nonsingular rational curve $E$ such that $E \cdot C=1$, the dimension of $|C|$ is the maximal dimension permitted by \cite{CC}, Theorem 1.1. This is consistent with the fact that linear systems of maximal dimension with respect to a fixed $g$ must be systems of hyperelliptic curves and dominate the hyperelliptic locus (see \cite{CC}, Proposition 2.3).
\begin{prop}
\label{hgn}
The $n$-pointed hyperelliptic locus $\mathcal{H}_{g,n}$ is uniruled for all $g \geq 2$ and $n \leq 4g+4$.
\end{prop}
\begin{proof}
Let $C$ be a general curve in $|\mathcal{O}(g+1,2)|$. Since $\mathbb{P}^1 \times \mathbb{P}^1$ is regular, it is sufficient to show that there exists a divisor in $|\mathcal{O}_C(C)|$ containing $n$ general points, $n \leq 4g+4$. Since the linear system is non-isotrivial, the statement will then follow from a straightforward adaptation of the proof of Theorem \ref{mgnuniruled}.\\
The cohomology sequence associated to the exact sequence
$$0 \rightarrow T_{\mathbb{P}^1 \times \mathbb{P}^1}(-C) \rightarrow T_{\mathbb{P}^1 \times \mathbb{P}^1} \rightarrow {T_{\mathbb{P}^1 \times \mathbb{P}^1}}_{|C} \rightarrow 0$$
gives $h^{0}({T_{\mathbb{P}^1 \times \mathbb{P}^1}}_{|C})=6+g$.\\
Since $h^{0}(T_{\mathbb{P}^1 \times \mathbb{P}^1})=6$ and $\dim W^3_{g+3}(C)=g$,
the linear series $|\mathcal{O}_C(1,1)|$, which is cut by hyperplanes in the projective embedding of $\mathbb{P}^1 \times \mathbb{P}^1$ as a quadric in $\mathbb{P}^3$, is general as a point in $W^3_{g+3}(C)$. Since $|\mathcal{O}_{C}(1,1)|=g^1_2+g^1_{g+1}$ is the sum of the two linear series defined on $C$, respectively, by the two projections of $\mathbb{P}^1 \times \mathbb{P}^1$,
we have that the $g^1_{g+1}$ is general too (the $g^1_2$ is unique).\\
Since $|\mathcal{O}_C(C)|=(g+1)g^1_2+2g^1_{g+1}$,
one has that $|\mathcal{O}_C(C)|$ is a general $g^{3g+4}_{4g+4}$, whose Brill-Noether number is $\rho=\rho(g,3g+4,4g+4)=g$. By Remark \ref{casod=r+rho} there is a divisor in $|\mathcal{O}_C(C)|$ containing $n$ general points, $n \leq 4g+4$.
\end{proof}
\end{section}
\begin{section}{Linear systems containing the general curve of genus $g$ on complete intersection surfaces}
\label{Linear systems containing the general curve of genus $g$ on complete intersection surfaces}
A natural question arising at this point is the following. Theorem \ref{mgnuniruled} is quite general and can in principle be applied to a wide range of linear systems, while we used it only on a few. Are there linear systems which can be used for an improvement of Theorem \ref{maintheorem}? If yes, which are the surfaces carrying them?\\
The results contained in Section \ref{background} tell that, up to a technical assumption, a general curve of genus $g \geq 12$ cannot move in a positive-dimensional non-isotrivial linear system on a nonsingular projective surface $S$ if $\text{Kod}(S) \leq 0$.\\
Hence such a linear system can be realized only on a surface of Kodaira dimension $1$ or on a surface of general type.\\
Now, it is a fact that positive-dimensional linear systems containing the general curve of genus $g$ for $12 \leq g \leq 15$ have been discovered on complete intersection surfaces. In \cite{CR} this is obvious since the curves are embedded in $\mathbb{P}^3$. In \cite{V} and \cite{BV} it does not come as a surprise once one looks at the proofs: typically a reducible curve, say $D_1 \cup D_2$, is constructed on a reducible complete intersection surface $Y$ in $\mathbb{P}^r$ and it is then proved that it can be smoothed to a general curve of genus $g$.\\
The whole construction is harder to handle if $Y$ is not a complete intersection.\\
In this section an extensive discussion on linear systems on complete intersection surfaces is carried out, giving a partial answer to the questions above.
It turns out that, if $S$ is a nonsingular complete intersection surface, it is not possible to have a general curve of genus $g \geq 16$ moving on it, while, if $12 \leq g \leq 15$, the surface $S$ must be a canonical one. There are only five families of canonical complete intersection surfaces. Two of these, namely the quintic surfaces in $\mathbb{P}^3$ and the complete intersections of four quadrics in $\mathbb{P}^6$, were exactly those used in Section \ref{unirulednessofsomemodulispaces} to prove our uniruledness results.\\ \\
It is a classical well-known fact that a general curve of genus $g \geq 12$ cannot be embedded in $\mathbb{P}^2$. The other families of complete intersection surfaces of negative Kodaira dimension are the quadric and the cubic surfaces in $\mathbb{P}^3$ and the complete intersections of two quadrics in $\mathbb{P}^4$ (all these are rational surfaces). The following proposition takes care of these cases:
\begin{prop}
\label{2322}
Let $C$ be a general curve of genus $g \geq 2$ moving in a positive-dimensional linear system on a nonsingular quadric or cubic surface in $\mathbb{P}^3$, or on a nonsingular complete intersection of two quadrics in $\mathbb{P}^4$. Then $g \leq \alpha$, where $\alpha=4,9,8$, respectively.
\end{prop}
\begin{proof}
Let $S=\mathbb{P}^1 \times \mathbb{P}^1$ and consider the exact sequence $0 \rightarrow T_C \rightarrow {T_S}_{|C} \rightarrow N_{C/S} \rightarrow 0$. The surface $S$ is rigid since $h^{1}(T_S)=0$, hence it is sufficient to show that the cohomology map $H^0(N_{C/S}) \xrightarrow{\beta} H^1(T_C)$ is not surjective for $g \geq 5$. One has $C \in |\mathcal{O}(a,b)|$, $a,b \geq 1$, $C^2=2ab$ and $g(C)=(a-1)(b-1)$, thus $h^1(N_{C/S})=h^1(\mathcal{O}_C(C))=0$. It follows that $\beta$ is not surjective if and only if $h^1({T_S}_{|C}) \neq 0$.\\
The cohomology sequence associated to the exact sequence $0 \rightarrow T_{S}(-C) \rightarrow T_{S} \rightarrow {T_{S}}_{|C} \rightarrow 0$ gives $h^{1}({T_S}_{|C})=h^{2}(T_{S}(-C))$.\\
Let $C \in |\mathcal{O}(a,b)|$. Using Serre duality one shows that $h^{2}(T_{S}(-C))=h^{0}(\mathcal{O}(a-4,b-2) \oplus \mathcal{O}(a-2,b-4))$, which equals $0$ if and only if $(a \leq 3 \lor b \leq 1) \wedge (a \leq 1 \lor b \leq 3)$. This condition is never satisfied if $a,b$ are such that $(a-1)(b-1)=g$ for $g \geq 5$ (while it is satisfied if $a=b=3$ i.e. for curves of genus $4$).\\
Let $S$ be a nonsingular cubic surface in $\mathbb{P}^3$ and let $d=\deg C$. Riemann-Roch formula gives $d \leq g+3$. Since $K_S =\mathcal{O}_S(-1)$, the genus formula gives $2g-1 \leq C^2 \leq 3g+1$. In particular $\mathcal{O}_C(C)$ is nonspecial and Riemann-Roch gives $h^{0}(\mathcal{O}_C(C))=C^2 +1-g$. Since $S$ is regular one has $h^{0}(\mathcal{O}_S(C))=C^2+2-g$. Substituting in inequality (\ref{5g-111chios}) one obtains
$$5(g-1)-11\chi(\mathcal{O}_S)+2K^2_S-C^2+2-g=4g-8-C^2 \leq 0.$$
Since $C^2 \leq 3g+1$, one has $4g-8-C^2 \geq g-9$, hence $g \leq 9$ must hold.\\
Let $S$ be a nonsingular complete intersection of two quadrics in $\mathbb{P}^4$. This time Riemann-Roch formula gives $d \leq g+4$ and the genus formula gives $2g-1 \leq C^2 \leq 3g+2$. Arguing as in the previous case one obtains
$$5(g-1)-11\chi(\mathcal{O}_S)+2K^2_S-C^2+2-g=4g-6-C^2 \leq 0.$$
Since $C^2 \leq 3g+2$, one has $4g-6-C^2 \geq g-8$, hence $g \leq 8$ must hold.\\
\end{proof}
There are only three families of complete intersection surfaces having Kodaira dimension $0$, namely quartics in $\mathbb{P}^3$, complete intersections of type $(2,3)$ in $\mathbb{P}^4$ and of type $(2,2,2)$ in $\mathbb{P}^5$ (all these are K3 surfaces). These cases are covered by Theorem \ref{kod=0leq11}, which gives the bound $g \leq 11$.\\
All remaining complete intersection surfaces are of general type, having very ample canonical bundle.
\begin{prop}
\label{elencoSditipogenerale}
Let $C$ be a general curve of genus $g \geq 3$ moving in a positive-dimensional linear system on a nonsingular projective surface $S$ of general type which is a complete intersection of $r-2$ hypersurfaces in $\mathbb{P}^r$. Then $S$ is a canonical surface, that is $S$ must be one of the following types:
\begin{itemize}
\item $S=(5)$;
\item $S=(2,4)$;
\item $S=(3,3)$;
\item $S=(2,2,3)$;
\item $S=(2,2,2,2)$.
\end{itemize}
Moreover, one has $g \leq 15$ and if $g \geq 12$, then $C$ is non-degenerate in $\mathbb{P}^r$.
\end{prop}
\begin{proof}
By assumption $S=(k_1,k_2,...,k_{r-2})$ is the nonsingular complete intersection of $r-2$ projective hypersurfaces of respective degree $k_1,...,k_{r-2}$ in $\mathbb{P}^r$.\\
Let $K_S=\mathcal{O}_S(h)$ and suppose by contradiction that $h \geq 2$. By adjunction one has $\mathcal{O}_C(K_C) \cong \mathcal{O}_C(h) \otimes \mathcal{O}_C(C)$. Since $C$ moves in a positive-dimensional linear system, the normal bundle $N_{C/S} \cong \mathcal{O}_C(C)$ has a nonzero section, hence the bundle $\mathcal{O}_C(h-2+C)$ has a nonzero section too, say $s$. Consider the exact sequence
$$0 \rightarrow \mathcal{O}_C(2) \xrightarrow{\otimes s} \mathcal{O}_C(K_C) \rightarrow T \rightarrow 0$$
where $T$ is the torsion sheaf supported on the zero locus of $s$. The associated cohomology sequence gives $h^1(\mathcal{O}_C(2)) \geq h^1(\mathcal{O}_C(K_C))$, thus $\mathcal{O}_C(2)$ is special. Since $h^{0}(\mathcal{O}_C(1)) \geq 2$, Proposition \ref{proprietàcurvemoduligenerali} gives a contradiction, hence $S$ must be a canonical surface.\\
Suppose by contradiction that $g(C) \geq 12$ and $C$ is a degenerate curve in $\mathbb{P}^r$. Then $C$ also lies on a nonsingular complete
intersection surface $S'$ of the kind $(k_1,...,k_{r-3},1)$ in $\mathbb{P}^r$ i.e. of the kind $(k_1,...,k_{r-3})$ in $\mathbb{P}^{r-1}$. Up to a permutation of the $k_i$, the surface $S'$ can then be supposed to be one among $(2), (3), (2,2), (2,2,2)$, which have Kodaira dimension $\leq 0$ (the case $S'=\mathbb{P}^2$ has already been ruled out).\\
Let $C^2_{S'}$ be the self-intersection of $C$ on $S'$. The genus formula gives $C^2_{S'} \geq 2g-2$. Since $S'$ is regular, $C$ moves on it in a positive-dimensional linear system. If $S'$ is a $(2,2,2)$ the contradiction is given by Theorem \ref{kod=0leq11}, otherwise it is given by Proposition \ref{2322}.\\
By contradiction let $g \geq 16$. By Theorem \ref{disuguaglianzafibrazionefree} inequality
\begin{equation}
\label{5g-111chios2}
5(g-1)-\left[11\chi(\mathcal{O}_S)-2K^2_S+2C^2\right]+h^{0}(\mathcal{O}_S(C)) \leq 0
\end{equation}
must be satisfied and $h^{0}(\mathcal{O}_S(C)) \geq 2$. Let $C$ be embedded in $S$ by a $g^{r}_{d}$ (by the previous part $C$ is nondegenerate in $\mathbb{P}^r$). Since $C$ has general moduli, the Brill-Noether number $\rho(g,r,d)=g-(r+1)(g-d+r)$ must be nonnegative, thus one has
\begin{equation}
\label{drr+1g+r}
d \geq \frac{r}{r+1}g+r
\end{equation}
from which
\begin{equation}
\label{c2casoS=(...)}
C^2=2g-2-C \cdot K_S=2g-2-d \leq \frac{r+2}{r+1}g-r-2.
\end{equation}
Let us examine each possible case.
\begin{itemize}
\item
Let $S$ be a nonsingular quintic in $\mathbb{P}^3$.
One has $5(g-1)-45-2C^2+h^{0}(\mathcal{O}_S(C)) \geq \frac{5}{2}g-40+h^{0}(\mathcal{O}_S(C)) > 0$ if $g \geq 16$, where the first term is the left term of (\ref{5g-111chios2}) and the first inequality follows by (\ref{c2casoS=(...)}). Thus $S$ cannot carry a general curve of genus $g \geq 16$ moving in a positive-dimensional linear system. The structure of the argument is the same for all the other cases.
\item Let $S=(2,4)$.
One has $5(g-1)-50-2C^2+h^{0}(\mathcal{O}_S(C)) \geq \frac{13}{5}g-43+h^{0}(\mathcal{O}_S(C)) > 0$ if $g \geq 16$, where the first term is the left term of (\ref{5g-111chios2}) and the first inequality follows by (\ref{c2casoS=(...)}).
\item Let $S=(3,3)$ in $\mathbb{P}^4$. One has $5(g-1)-48-2C^2+h^{0}(\mathcal{O}_S(C)) \geq \frac{13}{5}g-41+h^{0}(\mathcal{O}_S(C)) > 0$ if $g \geq 16$, where the first term is the left term of (\ref{5g-111chios2}) and the first inequality follows by (\ref{c2casoS=(...)}).
\item
Let $S=(2,2,3)$.
One has $5(g-1)-53-2C^2+h^{0}(\mathcal{O}_S(C)) \geq \frac{8}{3}g-44+h^{0}(\mathcal{O}_S(C)) >0$ if $g \geq 16$, where the first term is the left term of (\ref{5g-111chios2}) and the first inequality follows by (\ref{c2casoS=(...)}).
\item
Let $S=(2,2,2,2)$.
One has $5(g-1)-53-2C^2+h^{0}(\mathcal{O}_S(C)) \geq \frac{19}{7}g-45+h^{0}(\mathcal{O}_S(C)) >0$ if $g \geq 16$, where the first term is the left term of (\ref{5g-111chios2}) and the first inequality follows by (\ref{c2casoS=(...)}).
\end{itemize}
\end{proof}
One can sum up the previous results in the following theorem:
\begin{teo}
Let $C$ be a general curve of genus $g \geq 3$ moving in a positive-dimensional linear system on a nonsingular projective surface $S$ which is a complete intersection. Then $g \leq 15$. Moreover
\begin{itemize}
  \item if $g \leq 11$, then $\emph{Kod}(S) \leq 0$ or $S$ is a canonical surface;
  \item  if $12 \leq g \leq 15$, then $S$ is a canonical surface.
\end{itemize}
\end{teo}
This in particular implies that, if one wants to use Theorem \ref{equivalentimguniruledsistemalineare} to show uniru-ledness for some $\overline{\mathcal{M}}_g$, $g \geq 16$, the surface cannot be a complete intersection.
\end{section}
$\left.\right.$\\ \\
\textbf{Acknowledgements.}\\
The author wants to express his gratitude to his advisor Edoardo Sernesi which led him through to the world of moduli spaces of curves.
A special thank to Antonio Rapagnetta for many helpful discussions and to Claudio Fontanari for useful suggestions on the preliminary draft of this paper.

$\left.\right.$\\
$\left.\right.$\\
\noindent
Luca Benzo, Dipartimento di Matematica, Università di Roma Tor Vergata, Via della Ricerca Scientifica 1, 00133 Roma, Italy. e-mail
benzo@mat.uniroma2.it

\end{document}